\documentclass[11pt]{amsart}

\usepackage{epigamath}

\usepackage[english]{babel}

\numberwithin{equation}{section}

\usepackage{enumitem}

\usepackage[all]{xy}

\newtheorem{theorem}{Theorem}[section]
\newtheorem{corollary}[theorem]{Corollary}
\newtheorem{lemma}[theorem]{Lemma}
\newtheorem{proposition}[theorem]{Proposition}

\newtheorem{conjecture}[theorem]{Conjecture}

\theoremstyle{definition}

\newtheorem{remark}[theorem]{Remark}

\newcommand{\Mgn}{\overline{\mathcal{M}}_{g,n}}
\newcommand{\oM}{\overline{\mathcal{M}}}

\newcommand{\tpi}{\widetilde{\pi}}
\newcommand{\hpi}{\widehat{\pi}}
\newcommand{\gl}{\mathrm{gl}}

\def\DR{\mathrm{DR}}
\def\sfb{\mathsf{a}}
\def\sfc{\mathsf{c}}
\def\sfd{\mathsf{d}}
\def\sfe{\mathsf{e}}

\def\sfB{{\mathsf{B}}}

\newcommand{\refl}{\mathop\mathsf{\scriptstyle{(12)_*}}}

\EpigaVolumeYear{6}{2022} \EpigaArticleNr{8} \ReceivedOn{October 25,
2021}
\InFinalFormOn{January 31, 2022}
\AcceptedOn{February 22, 2022}

\title{A conjectural formula for $\DR_g(a,-a)\lambda_g$}
\titlemark{A conjectural formula for $\DR_g(a,-a)\lambda_g$}

\author{Alexandr Buryak}

\address{%
  Faculty of Mathematics,
  National Research University Higher School of Economics,
  6 Usacheva str., Moscow, 119048, Russian Federation;\newline
  Center for Advanced Studies,
  Skolkovo Institute of Science and Technology,
  1 Nobel str., Moscow, 143026, Russian Federation;\newline
  Faculty of Mechanics and Mathematics,
  Lomonosov Moscow State University,
  GSP-1, 119991 Moscow, Russian Federation
}
\email{aburyak@hse.ru}

\author{Francisco Hern{\'a}ndez Iglesias}
\address{%
  Korteweg-de Vries Instituut voor Wiskunde,
  Universiteit van Amsterdam, Postbus 94248, 1090GE Amsterdam, Nederland;\newline
  Institut de Math\'ematiques de Bourgogne,
  UMR 5584 CNRS, Universit\'e Bourgogne Franche-Comt\'e,
  F-2100 Dijon, France}
\email{f.hernandeziglesias@uva.nl}

\author{Sergey Shadrin}
\address{%
  Korteweg-de Vries Instituut voor Wiskunde,
  Universiteit van Amsterdam,
  Postbus 94248, 1090GE Amsterdam, Nederland}
\email{s.shadrin@uva.nl}

\authormark{A. Buryak, F. Hern{\'a}ndez Iglesias, and S. Shadrin}

\AbstractInEnglish{%
  We propose a conjectural formula for $\DR_g(a,-a)\lambda_g$ and
  check all its expected properties. Our formula refines the one point
  case of a similar conjecture made by the first named author in
  collaboration with Gu\'er\'e and Rossi, and we prove that the two
  conjectures are in fact equivalent, though in a quite non-trivial
  way.
}



\acknowledgement{The work of A.~B. (Sections 3.6 and 3.7) was supported by the grant
no.~20-11-20214 of the Russian Science Foundation. F.~H.~I. and
S.~S. were supported by the Netherlands Organization for Scientific
Research.}

\begin{document}



\maketitle

\begin{prelims}

\DisplayAbstractInEnglish

\bigskip

\DisplayKeyWords

\medskip

\DisplayMSCclass

\end{prelims}


\newpage

\setcounter{tocdepth}{1}

\tableofcontents


\section{Introduction}
In~\cite{Buryak} the first named author defined new Hamiltonian
integrable hierarchies, the so-called double ramification hierarchies,
associated to cohomological field theories. They are conjectured
in~\cite{Buryak} to be Miura equivalent to the Dubrovin--Zhang
hierarchies constructed in~\cite{DZ,BPS-1}. This conjecture is further
refined and made more explicit in~\cite{BDGR-1}, and in~\cite{BGR} it
is reduced to a system of conjectural relations between some
explicitly defined classes in the tautological ring of the moduli
space of curves $R^*(\oM_{g,n})$.

The one point case of the conjecture in~\cite{BGR} gives a
surprisingly simple expression for the product of the top Chern class
of the Hodge bundle $\lambda_g\in R^g(\oM_{g,1})$ and the push-forward
of the double ramification cycle $\DR_{{g}}(a,-a)\in R^{g}
(\oM_{g,2})$ under the map that forgets the second marked point. For
the definition of the double ramification cycle $\DR_{{g}}(a,-a)$ we
refer, for instance, to~\cite{BSSZ}, and for general information on
the tautological rings of the moduli spaces of curves to a recent
survey of Schmitt~\cite{Schmitt}.

In this paper we propose a refinement of the one point case of the
conjecture in~\cite{BGR}. We conjecture a formula for
$\DR_{{g}}(a,-a)\lambda_g \in R^{2g} (\oM_{g,2})$ in terms of a very
simple linear combination of natural strata equipped with psi classes
of the same type as in~\cite{BGR}. We analyze our formula in detail
and we prove that it satisfies virtually all properties one might
expect from the class $\DR_{{g}}(a,-a)\lambda_g$, including the
intersections with all natural boundary divisors in $\oM_{g,2}$ and
with the psi classes, and finally using these properties we also show
that our conjecture is in fact equivalent to the one point case of the
conjecture in~\cite{BGR}.

\subsection*{Organization of the paper}
In Section~\ref{sec:formulaandprop} we formulate our conjecture,
explain its relation to the one point case of the conjecture
in~\cite{BGR}, and state the expected properties of our formula. In
Section~\ref{sec:proofs} we introduce our main tools, a variety of
corollaries of the Liu--Pandharipande relations among the tautological
classes~\cite{Liu-Pand}, and prove all properties stated before. 

\subsection*{Acknowledgments} The authors thank D.~Holmes, P.~Rossi,
and J.~Schmitt for useful questions and correspondence. In particular,
J.~Schmitt has checked Conjecture~\ref{conj:main} below in genera $1$
and $2$, as well as in the Gorenstein quotient in genus $3$, with the
program \textsf{admcycles} \cite{admcycles}.

\section{Conjectural formula and its properties}
\label{sec:formulaandprop}

\subsection{Notation} 
Let $\oM_{g,n}$ be the Deligne--Mumford compactification of the moduli
space of curves with~$n$ marked points. There is a natural action of
the symmetric group $S_n$ on $\oM_{g,n}$ by relabeling the points. In
particular, for $n=2$ we will use the morphism that permutes the first
and second marked points that we denote by $\refl \colon
R^*(\oM_{g,2}) \to R^*(\oM_{g,2})$.

Let $\sigma\colon \oM_{g_1,2}\times \oM_{g_2,2} \to \oM_{g_1+g_2,2}$
glue the second marked point of $\oM_{g_1,2}$ and the first marked
point of $\oM_{g_2,2}$ into a node and identify the first marked point
in $\oM_{g_1,2}$ (respectively, the second marked point in
$\oM_{g_2,2}$) with the first (respectively, the second) marked point
in $\oM_{g_1+g_2,2}$. Let $c_1\in R^*(\oM_{g_1,2})$, $c_2\in
R^*(\oM_{g_2,2})$. It is convenient for us to denote throughout the
text $c_1\diamond c_2 \coloneqq \sigma_*(c_1\otimes c_2)$ and we use
$\diamond$ as an associative operation on classes in moduli spaces
with two marked points.

With the first two points distinguished, we can extend the notation
$\diamond$ to the push-forwards of the morphisms $\sigma\colon
\oM_{g_1,2}\times \oM_{g_2,2+n} \to \oM_{g_1+g_2,2+n}$ that glue the
second marked point of $\oM_{g_1,2}$ with the first marked point of
$\oM_{g_2,2+n}$ into a node and identify the first marked point in
$\oM_{g_1,2}$ (respectively, the second marked point in
$\oM_{g_2,2+n}$) with the first (respectively, the second) marked
point in $\oM_{g_1+g_2,2+n}$. We can do the same for the similar
morphisms $\sigma\colon \oM_{g_1,2+n}\times \oM_{g_2,2} \to
\oM_{g_1+g_2,2+n}$.

\subsection{Conjectural formula}
For $g_1,\ldots,g_k,g\ge 1$ and $d_1,\ldots,d_k\ge 0$ such that $\sum g_i=g$, let $\sfc^{g_1,\dots,g_k}_{d_1,\dots,d_k}\in R^{d_1+\cdots+d_k+k-1}(\oM_{g,2})$ be the class represented by the bamboo
\begin{equation*}
\vcenter{\xymatrix@C=10pt{
		&*+[o][F-]{g_1}\ar@{-}[l]*{{\ }_1\ }\ar@{-}[rr]^<<<<{\psi^{d_1}} &&*+[o][F-]{g_2}\ar@{-}[rr]^<<<<{\psi^{d_2}}&& *+[o][F-]{g_3}\ar@{--}[rrr]^<<<<{\psi^{d_3}} && & *+[o][F-]{g_k}\ar@{-}[r]*{{\ }_2}^<<<<{\psi^{d_k}} &
}}
= \psi_2^{d_1}|_{\oM_{g_1,2}} \diamond  \psi_2^{d_2}|_{\oM_{g_2,2}} \diamond \cdots \diamond  \psi_2^{d_k}|_{\oM_{g_k,2}}.
\end{equation*}
Denote
\begin{align*}
	\overset{\rightharpoondown}{\sfc}^g_{d\mathop{|}k} & \coloneqq  \sum_{\substack{g_1,\dots,g_k\\d_1,\dots,d_k}} \sfc^{g_1,\dots,g_k}_{d_1,\dots,d_k}\in R^d(\oM_{g,2}),
\end{align*}
where the sum is taken over all $g_1+\cdots+g_k=g$ and all $d_1+\dots+d_k+k-1=d$ satisfying the inequalities
\[
  d_1+\cdots+d_\ell + \ell -1  \leq 2(g_1+\cdots+g_\ell) -1,\quad \ell = 1,\dots,k.
\]
Note that by the definition 
\begin{equation}\label{eq:c-class vanishes}
  \overset{\rightharpoondown}{\sfc}^g_{d\mathop{|}k}
  =0
  \qquad\text{if $k>g$ or $d\ge 2g$}.
\end{equation}
Let
\begin{equation}
  \label{eq:Bg-definition}
  \begin{split}
    \sfB^g
    &\coloneqq \psi_2^{2g}|_{\oM_{g,2}} + \sum_{\substack{g_1+g_2=g \\
        d_1+d_2=2g-1}} \sum_{k=1}^{g_1} (-1)^k
    \overset{\rightharpoondown}{\sfc}^{g_1}_{d_1\mathop{|}k}
    \diamond \psi_2^{d_2}|_{\oM_{g_2,2}} \\
    &\phantom{:}=
    \sum_{k=1}^g (-1)^{k-1} \sum_{\substack{d_1,\dots,d_k \\
        g_1,\dots,g_k}} \psi_2^{d_1}|_{\oM_{g_1,2}}
    \diamond  \psi_2^{d_2}|_{\oM_{g_2,2}}
    \diamond \cdots \diamond
    \psi_2^{d_k}|_{\oM_{g_k,2}}\in R^{2g}(\oM_{g,2}),    
  \end{split}
\end{equation}
where the last sum is taken over all $g_1+\cdots+g_k=g$,
$g_1,\dots,g_k\geq 1$, and $d_1+\cdots+d_k+k-1 = 2g$,
$d_1,\dots,d_k\geq 0$, with the extra condition that for any $1\leq
\ell\leq k-1$ we have $d_1+\cdots+d_\ell +\ell-1\leq 2(g_1+\cdots
+g_\ell) -1$.

\begin{conjecture}\label{conj:main}
  We have $a^{-2g} \DR_{{g}}(a,-a)\lambda_g = \sfB^g$. 
\end{conjecture}

Note that the left-hand side of this equation can be expressed in the
tautological classes using the formula of
Janda--Pandharipande--Pixton--Zvonkine~\cite{JPPZ} or, taking into
account the factor~$\lambda_g$, it is sufficient to use the Hain
formula~\cite{Hain} (see an explanation, e.g., in~\cite[Section
2]{Buryak-Rossi}). However, the resulting expressions are much more
complicated than the one we conjecture here. Observe that the
right-hand side is independent of $a$, which is consistent with Hain's
formula, which states that the compact-type part of $\DR_g(a, -a)$ is
a homogeneous polynomial in $a$ of degree $2g$. Thus, it is enough to
prove the conjecture for the case $a = 1$.

\subsection{Relation to an earlier conjecture for the push-forwards}
Conjecture~\ref{conj:main} is a refinement of the one point case of a conjecture of the first named author with Gu\'er\'e and Rossi~\cite[Conjecture~2.5]{BGR}. Indeed, recall the definition of the class $B^g_{2g-1}\in R^{2g-1}(\oM_{g,1})$ in~\cite{BGR}. We have:
\begin{align*} 
B^g_{2g-1}\coloneqq \sum_{k=1}^g (-1)^{k-1} \sum_{\substack{g_1,\dots,g_k\\a_1,\dots,a_k}} \vcenter{\xymatrix@C=10pt{
		&*+[o][F-]{g_1}\ar@{-}[rr]^<<<<{\psi^{a_1}} &&*+[o][F-]{g_2}\ar@{-}[rr]^<<<<{\psi^{a_2}}&& *+[o][F-]{g_3}\ar@{--}[rrr]^<<<<{\psi^{a_3}} && & *+[o][F-]{g_k}\ar@{-}[r]*{{\ }_1}^<<<<{\psi^{a_k}} &
}},
\end{align*}
where the sum is taken over all $g_1+\cdots+g_k=g$, $g_1,\dots,g_k\geq 1$, and $a_1,\dots,a_k\geq 0$ such that
\[a_1+\cdots+a_k+k-1 = 2g-1\ \text{and}\ a_1+\cdots+a_\ell + \ell-1 \leq 2(g_1+\cdots+g_\ell) -2\ \text{for}\ \ell=1,\dots,k-1.\]

Let $\pi\colon \oM_{g,2}\to \oM_{g,1}$ be the map that forgets the second marked point. In the one point case the conjecture from~\cite[Conjecture 2.5]{BGR} is reduced to the identity 
$$
a^{-2g} \pi_*( \DR_{{g}}(a,-a)\lambda_g) = B^g_{2g-1},
$$
see~\cite[Section 4.2]{BGR}. On the other hand, we have the following statement.

\begin{proposition}\label{prop:reltoBGR} We have $\pi_* \refl \sfB^g = B^g_{2g-1}$. 
\end{proposition} 

\begin{proof} It follows from the fact that $\pi_*(\psi_1^d) = \psi_1^{d-1}$ for $d\geq 1$ and $\pi_*(\psi_1^0) = 0$. Thus all terms with $d_1=0$ in~\eqref{eq:Bg-definition} vanish under the push-forward, and all other terms are in one-to-one correspondence with $a_1=d_1-1$ and $a_i=d_i$ for $i=2,\dots,k$, $k=1,\dots,g$. 
\end{proof}

\begin{remark}
Note that an expected property of $a^{-2g}\DR_{{g}}(a,-a)\lambda_g$ is that it is invariant under $\refl$, and indeed we prove below that $\refl\sfB^g = \sfB^g$, $g\geq 1$, so in fact we can reformulate the statement of Proposition~\ref{prop:reltoBGR} as $\pi_* \sfB^g = B^g_{2g-1}$. 	
\end{remark}

In fact, it is also possible to prove a much stronger statement than Proposition~\ref{prop:reltoBGR}. 

\begin{theorem}\label{thm:equivalent} The two conjectural formulas, $a^{-2g} \pi_*( \DR_{{g}}(a,-a)\lambda_g) = B^g_{2g-1}$ and $a^{-2g} \DR_{{g}}(a,-a)\lambda_g = \sfB^g$, are equivalent. 
\end{theorem}

The first formula follows from the second one by Proposition~\ref{prop:reltoBGR}. The  implication in the other direction is quite non-trivial, and we postpone its proof until Section~\ref{sec:proofequiv}.

\subsection{Properties} We write down a list of properties of $\sfB^g$. 

\begin{theorem} We have:
\begin{align}
	\label{eq:symmetric}
	 & \refl \sfB^g = \sfB^g; 
\\ \label{eq:Thm-irrdiv}
	 & \sfB^g \cdot \vcenter{
	 	\xymatrix@M=2pt@C=10pt@R=0pt{
	 		& &  \\ 
	 		&*+[o][F-]{{g\text{-}1}}\ar@{-}[lu]*{{\ }_1\ }\ar@{-}[ld]*{{\ }_2\ }\ar@{-}@(ur,dr)\ar@{-}@(dr,ur) & \\ & &
	 	}
 	}= 0; 
\\ \label{eq:Thm-20div}
 	 & \sfB^g \cdot \vcenter{
 	 	\xymatrix@C=10pt@R=0pt{
 	 		& & \\ 
 	 		&*+[o][F-]{{g_1}}\ar@{-}[lu]*{{\ }_1\ }\ar@{-}[ld]*{{\ }_2\ }\ar@{-}[r] &*+[o][F-]{{g_2}}\\ & & 
 	 	}
 	 }  = 0, & & g_1+g_2 = g, \ g_2\geq 1; 
\\ \label{eq:Thm-Inter-1-g1-g2-2}
& \sfB^g \cdot \vcenter{
\xymatrix@C=10pt@R=10pt{
	&*+[o][F-]{{g_1}}\ar@{-}[l]*{{\ }_1\ }\ar@{-}[r]& *+[o][F-]{{g_2}} \ar@{-}[r]*{{\ }_2}& 
	}
} = \sfB^{g_1}\diamond \sfB^{g_2}, 
& & g_1+g_2 = g, \ g_1,g_2\geq 1;  \\
& \pi^*(\sfB^g) \cdot \psi_1 = \sfB^g\diamond 1|_{\oM_{0,3}} +\sum_{\substack{g_1+g_2=g\\ g_1,g_2\geq 1}} \sfB^{g_1}\diamond \pi^*(\sfB^{g_2}); \label{eq:thm-pullback}\\
& \sfB^g \cdot \psi_1 = \sum_{\substack{g_1+g_2=g\\ g_1,g_2\geq 1}} \frac{g_2}{g}\sfB^{g_1}\diamond \sfB^{g_2}, \label{eq:Thm-ev-psiclass}
\end{align}
where $\pi\colon \oM_{g,3}\to\oM_{g,2}$, $g\geq 1$, in~\eqref{eq:thm-pullback} is the projection that forgets the third marked point. 
\end{theorem}

The proof of this theorem is given in Section~\ref{sec:proofs}.

\begin{remark}
All these properties are satisfied by $\DR_{{g}}(1,-1)\lambda_g=a^{-2g}\DR_{{g}}(a,-a)\lambda_g$, namely:
\begin{itemize}
	\item $\DR_g(1,-1) \lambda_g  = \DR_g(-1,1)  \lambda_g $ is immediate from Hain's formula \cite{Hain}.
	\item $\DR_g(1, -1) \lambda_g  \cdot \vcenter{
		\xymatrix@M=2pt@C=10pt@R=0pt{
			& &  \\ 
			&*+[o][F-]{{g\text{-}1}}\ar@{-}[lu]*{{\ }_1\ }\ar@{-}[ld]*{{\ }_2\ }\ar@{-}@(ur,dr)\ar@{-}@(dr,ur) & \\ & &
		}
	} = 0$, as $\lambda_g$ restricts to zero on $\vcenter{
	\xymatrix@M=2pt@C=10pt@R=0pt{
		& &  \\ 
		&*+[o][F-]{{g\text{-}1}}\ar@{-}[lu]*{{\ }_1\ }\ar@{-}[ld]*{{\ }_2\ }\ar@{-}@(ur,dr)\ar@{-}@(dr,ur) & \\ & &
	}
} $.
\item $\DR_g(1,-1) \lambda_g \cdot \vcenter{
	\xymatrix@C=10pt@R=0pt{
		& & \\ 
		&*+[o][F-]{{g_1}}\ar@{-}[lu]*{{\ }_1\ }\ar@{-}[ld]*{{\ }_2\ }\ar@{-}[r] &*+[o][F-]{{g_2}}\\ & & 
	}
}  = 0$, as the classes $\DR_g(1,-1)$ and $\lambda_g$ respectively restrict to $\DR_{g_1}(1,-1,0)\otimes \DR_{g_2}(0)$ and $\lambda_{g_1}\otimes \lambda_{g_2}$ on $\overline{\mathcal{M}}_{g_1, 3} \times \overline{\mathcal{M}}_{g_2, 1}$. The vanishing follows after observing $\DR_{g_2}(0)\lambda_{g_2} = (-1)^{g_2} \lambda_{g_2}^2 =  0$.
\item $\DR_g(1,-1) \lambda_g \cdot \vcenter{
	\xymatrix@C=10pt@R=10pt{
		&*+[o][F-]{{g_1}}\ar@{-}[l]*{{\ }_1\ }\ar@{-}[r]& *+[o][F-]{{g_2}} \ar@{-}[r]*{{\ }_2}& 
	}
} = \DR_{g_1}(1,-1) \lambda_{g_1} \diamond \DR_{g_2}(1,-1) \lambda_{g_2}$, as $\DR_g(1,-1)$ and $\lambda_g$ respectively restrict to $\DR_{g_1}(1,-1) \otimes \DR_{g_2}(1,-1)$ and $\lambda_{g_1} \otimes \lambda_{g_2}$ on $\overline{\mathcal{M}}_{g_1, 2} \times \overline{\mathcal{M}}_{g_2, 2}$. 

\item The equality
\[\pi^*(\DR_g(1,-1)\lambda_g) \cdot \psi_1 = \DR_g(1,-1)\lambda_g\diamond 1|_{\oM_{0,3}} +\sum_{\substack{g_1+g_2=g\\ g_1,g_2\geq 1}} \DR_{g_1}\lambda_{g_1}\diamond \pi^*(\DR_{g_2}\lambda_{g_2})\]
follows from \cite[Theorem~5]{BSSZ}: one should use that $\pi^*\DR_g(1,-1)=\DR_g(1,-1,0)$, apply the formula of \cite[Theorem~5]{BSSZ} with $s=1$ and $n=l=3$, and then multiply the result by $\lambda_g$ noting that the terms with $p\ge 2$ will vanish after that.

\item The identity
\[\DR_g(1,-1)\lambda_g \cdot \psi_1 = \sum_{\substack{g_1+g_2=g\\ g_1,g_2\geq 1}} \frac{g_2}{g}\DR_{g_1}(1,-1)\lambda_{g_1}\diamond \DR_{g_2}(1,-1)\lambda_{g_2}\]
follows from the formula of~\cite[Theorem~4]{BSSZ} multiplied by $\lambda_g$, where one should again note that the terms with $p\ge 2$ vanish after this multiplication.
\end{itemize}
\end{remark}

\section{Proofs} \label{sec:proofs}

\subsection{Liu--Pandharipande relations}\label{sec:LiuPand}
Fix sets of indices $I_1$ and $I_2$ such that $I_1\sqcup I_2 = \{1,\dots,n\}$. Let $\Delta_{g_1,g_2}\subset \Mgn$ denote the divisor in $\Mgn$ whose generic points are represented by two-component curves intersecting at a node, where the two components have genera $g_1,g_2$ and contain the points with the indices $I_1,I_2$, respectively. 
Note that if $g_i=0$, then $|I_i|$ must be at least $2$, for the stability condition. 

For each $\Delta_{g_1,g_2}$ we consider the map $\iota_{g_1,g_2}\colon \oM_{g_1,|I_1|+1}\times  \oM_{g_2,|I_2|+1}\to\oM_{g,n}$ that glues the last marked points into a node and whose image is $\Delta_{g_1,g_2}$. Let $\psi_{\circ1}$ (respectively, $\psi_{\circ2}$) denote the psi classes at the marked points on the first (respectively, second) component that are glued into the node. 

\begin{proposition}[{\cite[Proposition 1]{Liu-Pand}}] For any $g\geq 0$, $n\geq 4$, $I_1$ and $I_2$ such that $I_1\sqcup I_2 = \{1,\dots,n\}$ and $|I_1|,|I_2|\geq 2$, and an arbitrary $r\geq 0$ we have:
\begin{equation} \label{eq:Liu-Pand-main}
	\sum_{\substack{g_1,g_2\geq 0\\ g_1+g_2=g}}	\sum_{\substack{a_1,a_2\geq 0\\ a_1+a_2=\\ 2g-3+n+r}} (-1)^{a_1} (\iota_{g_1,g_2})_*\psi_{\circ1}^{a_1}\psi_{\circ2}^{a_2} = 0 \in R^{2g-2+n+r}(\oM_{g,n}).
\end{equation}
\end{proposition}

This relation has the following corollaries. 
	
\begin{corollary}
  \label{cor:LP1}
  For any $g\geq 1$, $n\geq 1$, $r\geq 0$ we have:
\begin{equation} \label{eq:Liu-Pand-cor1-1}
  (-1)^{2g+n+r}\psi_2^{2g+n+r}
  +\sum_{\substack{g_1\geq 0,\, g_2>0\\ g_1+g_2=g}}
  \sum_{\substack{a_1,a_2\geq 0\\ a_1+a_2= 2g-1+n+r}} (-1)^{a_1} 
  \psi_2^{a_1}|_{\oM_{g_1,2+n}}\diamond \psi_1^{a_2}|_{\oM_{g_2,2}}
  = 0 \in R^{2g+n+r}(\oM_{g,n+2})
\end{equation}
and 
\begin{equation} \label{eq:Liu-Pand-cor1-2}
  -\psi_1^{2g+n+r} +\sum_{\substack{g_1>0,\, g_2\geq 0\\ g_1+g_2=g}}
  \sum_{\substack{a_1,a_2\geq 0\\ a_1+a_2= 2g-1+n+r}} (-1)^{a_1} 
  \psi_2^{a_1}|_{\oM_{g_1,2}}\diamond \psi_1^{a_2}|_{\oM_{g_2,2+n}}
  = 0 \in R^{2g+n+r}(\oM_{g,n+2}).
\end{equation}
\end{corollary}

\begin{corollary}[{\cite[Proposition 2]{Liu-Pand}}]
  \label{cor:LP2}
  For any $g\geq 1$, $r\geq 0$ we have:
  \begin{equation} \label{eq:Liu-Pand-cor2}
    -\psi_1^{2g+r} + (-1)^{2g+r}\psi_2^{2g+r}
    +\sum_{\substack{g_1,g_2>0 \\ g_1+g_2=g}}
    \sum_{\substack{a_1,a_2\geq 0 \\ a_1+a_2= 2g-1+r}} (-1)^{a_1} 
    \psi_2^{a_1}|_{\oM_{g_1,2}}\diamond \psi_1^{a_2}|_{\oM_{g_2,2}}
    = 0  \in R^{2g+r}(\oM_{g,2}).
  \end{equation}
\end{corollary}

All corollaries are proved by taking suitable push-forwards of the
relations~\eqref{eq:Liu-Pand-main} under the maps forgetting the
marked points, see~\cite{Liu-Pand,HI-S}.

\subsection{The symmetry property} \label{sec:symm}

Denote $\overset{\leftharpoondown}{\sfc}^g_{d \mathop{|} k}: =
\refl\overset{\rightharpoondown}{\sfc}^g_{d \mathop{|} k} $. We will
use the following conventions to simplify notation:
\[
  \overset{\rightharpoondown}{\sfc}^0_{-1 \mathop{|} 0} \diamond
  \psi^d_2|_{\overline{\mathcal{M}}_{g,2}}: =
  \psi^d_2|_{\overline{\mathcal{M}}_{g,2}}
  \quad\text{and}\quad
  \psi_1^d|_{\overline{\mathcal{M}}_{g,2}} \diamond
  \overset{\leftharpoondown}{\sfc}^0_{-1 \mathop{|} 0} :=
  \psi_1^d|_{\overline{\mathcal{M}}_{g,2}}.
\]
Note that there is the
following recursion relation for the classes
$\overset{\rightharpoondown}{\sfc}^g_{d \mathop{|} k}$: 
\begin{align*}
\overset{\rightharpoondown}{\sfc}^g_{d \mathop{|} k+1}=\sum_{\substack{g_1+g_2=g\\d_1+d_2=d-1}}\overset{\rightharpoondown}{\sfc}^{g_1}_{d_1 \mathop{|} k}\diamond \psi_2^{d_2}|_{\oM_{g_2,2}},\quad k\ge 0, \ d\le 2g-1,
\end{align*}
where $d_1 = -1$ is allowed in the sum to include the case $k=0$, as explained before. Let us now prove Equation~\eqref{eq:symmetric}. Let 
\begin{align*}
E :=   \sfB^g - \refl \sfB^g = \sum_{\substack{g_1+g_2=g \\ d_1+d_2=2g-1}} \sum_{k=0}^{g_1} (-1)^k \overset{\rightharpoondown}{\sfc}^{g_1}_{d_1\mathop{|}k} \diamond \psi_2^{d_2}|_{\oM_{g_2,2}} - \sum_{\substack{g_1+g_2=g \\ d_1+d_2=2g-1}}  \sum_{k=0}^{g_2} (-1)^k \psi_1^{d_1}|_{\oM_{g_1,2}}
\diamond \overset{\leftharpoondown}{\sfc}^{g_2}_{d_2\mathop{|}k}.
\end{align*}
Let $E_\ell$ denote the terms of $E$ consisting of exactly $\ell$ components, \emph{i.e.}, 
\begin{align*}
E_{\ell} :=  
\sum_{\substack{g_1+g_2=g \\ d_1+d_2=2g-1}}  (-1)^{\ell-1} \overset{\rightharpoondown}{\sfc}^{g_1}_{d_1\mathop{|} \ell - 1} \diamond \psi_2^{d_2}|_{\oM_{g_2,2}}
-
\sum_{\substack{g_1+g_2=g \\ d_1+d_2=2g-1}}  (-1)^{\ell-1} \psi_1^{d_1}|_{\oM_{g_1,2}}
\diamond \overset{\leftharpoondown}{\sfc}^{g_2}_{d_2\mathop{|} \ell -1}.
\end{align*}

\begin{lemma} \label{lem:intermediate-symmetry} We can write $E_1 + \dots + E_\ell$ as an expression involving only graphs with $\ell+1$ vertices. In particular:
	\begin{align*}		 
		 E_1 + \dots + E_\ell = (-1)^{\ell+1} \sum_{\substack {r+s = \ell-1 \\ g_1+g_2+g_3 + g_4 = g \\ d_1+d_2+d_3 + d_4 = 2g-3}}  (-1)^{d_1+d_2} \overset{\rightharpoondown}{\sfc}^{g_1}_{d_1\mathop{|}r}\diamond
		 \psi_2^{d_2}|_{\oM_{g_2,2}} \diamond \psi_1^{d_3}|_{\oM_{g_3,2}} 
		 \diamond \overset{\leftharpoondown}{\sfc}^{g_4}_{d_4\mathop{|}s}.
	\end{align*}
\end{lemma}
\begin{proof} We prove the lemma by induction. The base of induction is the $\ell=1$ case, which follows immediately from \eqref{eq:Liu-Pand-cor2}:
\begin{align*}
		E_1 = \psi_2^{2g}|_{\oM_{g,2}} - \psi_1^{2g}|_{\oM_{g,2}} =  - \sum_{\substack {g_1+g_2 = g \\ d_1+d_2 = 2g-1}}  (-1)^{d_1} \overset{\rightharpoondown}{\sfc}^{0}_{-1 \mathop{|} 0}\diamond
		\psi_2^{d_1}|_{\oM_{g_1,2}} \diamond \psi_1^{d_2}|_{\oM_{g_2,2}} 
		\diamond \overset{\leftharpoondown}{\sfc}^{0}_{-1 \mathop{|} 0}.
\end{align*}
In order to prove the step of induction, assume the lemma is true for $\ell \geq 1$. Then
\begin{equation}
  \label{eq:symmetryinductivestatement}
  \begin{split}
    E_1 + \dots + E_{\ell + 1}
    &= (-1)^{\ell+1}
    \sum_{\substack {%
        r+s = \ell-1 \\
        g_1+g_2+g_3 + g_4 = g \\
        d_1+d_2+d_3 + d_4 = 2g-3}
    } (-1)^{d_1+d_2} \overset{\rightharpoondown}{\sfc}^{g_1}_{d_1\mathop{|}r}\diamond
    \psi_2^{d_2}|_{\oM_{g_2,2}} \diamond \psi_1^{d_3}|_{\oM_{g_3,2}} 
    \diamond \overset{\leftharpoondown}{\sfc}^{g_4}_{d_4\mathop{|}s} \\
    &\quad+ \sum_{\substack{g_1+g_2=g \\ d_1+d_2=2g-1}}
    (-1)^{\ell} \overset{\rightharpoondown}{\sfc}^{g_1}_{d_1\mathop{|} \ell }
    \diamond \psi_2^{d_2}|_{\oM_{g_2,2}}
    - \sum_{\substack{g_1+g_2=g \\ d_1+d_2=2g-1}}
    (-1)^{\ell} \psi_1^{d_1}|_{\oM_{g_1,2}}
    \diamond \overset{\leftharpoondown}{\sfc}^{g_2}_{d_2\mathop{|} \ell }.    
  \end{split}
\end{equation}
We can split the first summand into two in the following way:
\begin{itemize}
	\item $d_1 + d_2 \leq 2(g_1 + g_2) -2$ and $d_3 + d_4 \geq 2(g_3 + g_4) - 1$;
	\item $d_1 + d_2 \geq 2(g_1 + g_2) -1$ and $d_3 + d_4 \leq 2(g_3 + g_4) - 2$.
\end{itemize}
Thus, the summand with $d_1 + d_2\leq 2(g_1 + g_2) -2$ takes the form 
\begin{align*}
  (-1)^{\ell+1}
  \sum_{\substack {%
      r+s = \ell-1 \\
      g_1+g_2+g_3 + g_4 = g \\
      d_1+d_2+d_3 + d_4 = 2g-3 \\
      d_1 + d_2 \leq 2(g_1 + g_2) - 2}
  }  &(-1)^{d_1+d_2}
  \overset{\rightharpoondown}{\sfc}^{g_1}_{d_1\mathop{|}r}\diamond
  \psi_2^{d_2}|_{\oM_{g_2,2}} \diamond \psi_1^{d_3}|_{\oM_{g_3,2}} 
  \diamond \overset{\leftharpoondown}{\sfc}^{g_4}_{d_4\mathop{|}s} \\
  &= (-1)^{\ell}
  \sum_{\substack {%
      r+s = \ell,\,r \geq 1 \\
      g_1+g_2+g_3 = g \\
      d_1+d_2+d_3 = 2g-2}
  }  (-1)^{d_1} \overset{\rightharpoondown}{\sfc}^{g_1}_{d_1\mathop{|}r}
  \diamond \psi_1^{d_2}|_{\oM_{g_2,2}} 
  \diamond \overset{\leftharpoondown}{\sfc}^{g_3}_{d_3\mathop{|}s}.
\end{align*}
Note that the third summand of \eqref{eq:symmetryinductivestatement}
corresponds to the missing terms with $r = 0$ in the last
expression. Similarly, for the terms with 
$d_3 + d_4\leq 2(g_3 + g_4) - 2$
\begin{align*}
  (-1)^{\ell+1} \sum_{\substack {r+s = \ell-1 \\
  g_1+g_2+g_3 + g_4 = g \\
  d_1+d_2+d_3 + d_4 = 2g-3 \\
  d_3 + d_4 \leq 2(g_3 + g_4) - 2}}
  &(-1)^{d_1+d_2} \overset{\rightharpoondown}{\sfc}^{g_1}_{d_1\mathop{|}r}\diamond
    \psi_2^{d_2}|_{\oM_{g_2,2}} \diamond \psi_1^{d_3}|_{\oM_{g_3,2}} 
    \diamond \overset{\leftharpoondown}{\sfc}^{g_4}_{d_4\mathop{|}s} \\
  &= - (-1)^{\ell} \sum_{\substack {r+s = \ell,\,s \geq 1 \\ g_1+g_2+g_3 = g \\ d_1+d_2+d_3 = 2g-2}}  (-1)^{d_1+d_2} \overset{\rightharpoondown}{\sfc}^{g_1}_{d_1\mathop{|}r}\diamond
\psi_2^{d_2}|_{\oM_{g_2,2}} 
\diamond \overset{\leftharpoondown}{\sfc}^{g_3}_{d_3\mathop{|}s}.
\end{align*}
Again, the second summand of \eqref{eq:symmetryinductivestatement} corresponds to the missing terms with $s = 0$ in the last expression. Putting everything together
\begin{align*}
E_1 + \dots + E_{\ell+1} = (-1)^{\ell} \sum_{\substack {r+s = \ell \\ g_1+g_2+g_3 = g \\ d_1+d_2+d_3 = 2g-2 }}  (-1)^{d_1} \overset{\rightharpoondown}{\sfc}^{g_1}_{d_1\mathop{|}r}
\diamond \left.\left( \psi_1^{d_2} -(-1)^{d_2} \psi_2^{d_2} \right)\right|_{\oM_{g_2,2}}
\diamond \overset{\leftharpoondown}{\sfc}^{g_3}_{d_3\mathop{|}s}.
\end{align*}
Using~\eqref{eq:c-class vanishes} we see that a term in the last sum is equal to zero unless $d_2\ge 2g_2$. Then the result follows after applying \eqref{eq:Liu-Pand-cor2} to the last expression.
\end{proof}

Applying the lemma above to $E = E_1 + \dots + E_g$ proves equation
\eqref{eq:symmetric}.

\subsection{Intersections with divisors of two types}\label{sec:proof-irrdev} 

Here we prove Equations~\eqref{eq:Thm-irrdiv} and~\eqref{eq:Thm-20div}. It is convenient to use the following notations for the classes of the divisors under consideration:
\begin{align*}
\delta_g\coloneqq  \vcenter{
		\xymatrix@M=2pt@C=10pt@R=0pt{
			& &  \\ 
			&*+[o][F-]{{g\text{-}1}}\ar@{-}[lu]*{{\ }_1\ }\ar@{-}[ld]*{{\ }_2\ }\ar@{-}@(ur,dr)\ar@{-}@(dr,ur) & \\ & &
		}
	},\qquad	
	\delta'_{g} \coloneqq 
\begin{cases}
	\vcenter{
	\xymatrix@C=10pt@R=0pt{
		& & \\ 
		&*+[o][F-]{{g\text{-}h}}\ar@{-}[lu]*{{\ }_1\ }\ar@{-}[ld]*{{\ }_2\ }\ar@{-}[r] &*+[o][F-]{{h}}\\ & & 
	}
}, &\text{if $g\geq h$}; \\
0, &\text{if $g<h$},
\end{cases}	
\end{align*}
where we fixed $h\ge 1$. So let us prove that 
\begin{equation}\label{eq:omega-B}
\omega_g B^g=0\quad\text{if}\quad\text{$\omega_g=\delta_g$ or $\omega_g=\delta'_g$}.
\end{equation}

We will use the following property: $\omega_{g_1+g_2}(\alpha\diamond\beta)=\omega_{g_1}\alpha\diamond\beta+\alpha\diamond\omega_{g_2}\beta$, where $\alpha\in R^*(\oM_{g_1,2})$ and $\beta\in R^*(\oM_{g_2,2})$. We decompose
\begin{align*}
\omega_g \sfB^g = \sum_{k\ge 1}E_k,\quad\text{where}\quad E_k := (-1)^{k-1} \sum_{\substack{ g_1+g_2 = g \\ d_1 + d_2 = 2g-1}}  \omega_g\left( \overset{\rightharpoondown}{\sfc}^{g_1}_{d_1 \mathop{|} k-1} \diamond \psi_2^{d_2}|_{\overline{\mathcal{M}}_{g_2,2}}\right).
\end{align*}

\begin{lemma} \label{lem:intermediate-irrdiv} We have:
	\begin{align}\label{eq:divisors of two types}
	E_1 + \dots + E_k &= (-1)^{k+1} \sum_{\substack{g_1 + g_2 + g_3 + g_4 = g \\ d_1 + d_2 + d_3 + d_4 = 2g-3 \\ r + s = k - 1}} (-1)^{d_1 + d_2} \omega_{g_1+g_2}\left(\overset{\rightharpoondown}{\sfc}^{g_1}_{d_1 \mathop{|} r} \diamond \psi^{d_2}_2|_{\overline{\mathcal{M}}_{g_2, 2}}\right) \diamond \psi_1^{d_3}|_{\overline{\mathcal{M}}_{g_3,2}} \diamond \overset{\leftharpoondown}{\sfc}^{g_4}_{d_4 \mathop{|} s}.
	\end{align}
\end{lemma}
\begin{proof} We prove the lemma by induction. The base of induction is the $k=1$ case, which reads:
	\begin{align*}
	E_1 = \omega_g \psi_2^{2g}|_{\overline{\mathcal{M}}_{g,2 }} = - \sum_{\substack { g_1+g_2 = g \\ d_1 + d_2 = 2g-1}} (-1)^{d_1} \omega_{g_1} \psi_2^{d_1}|_{\overline{\mathcal{M}}_{g_1,2}} \diamond \psi_1^{d_2}|_{\overline{\mathcal{M}}_{g_2,2}}.
	\end{align*}
If $\omega_i=\delta_i$, then this equation follows from the genus $g-1$ case of \eqref{eq:Liu-Pand-cor1-1} with $n=2$ and $r=0$, after taking the push-forward under the map that glues two marked points. If $\omega_i=\delta'_i$, then the equation follows from the genus $g-h$ case of \eqref{eq:Liu-Pand-cor1-1} with $n=1$ and an appropriate $r$, after taking the product with $1\in R^0(\oM_{h,1})$ and then the push-forward under the gluing map $\oM_{g-h,3}\times\oM_{h,1}\to\oM_{g,2}$.

Let us now assume that the lemma holds for $k \geq 1$. Then for $k + 1$ we have
\begin{align}\label{eq:circlelineEk+1}
E_1 + \dots + E_{k+1} =& (-1)^{k+1} \sum_{\substack{g_1 + g_2 + g_3 + g_4 = g \\ d_1 + d_2 + d_3 + d_4 = 2g-3 \\ r + s = k - 1}} (-1)^{d_1 + d_2} \omega_{g_1+g_2}\left(\overset{\rightharpoondown}{\sfc}^{g_1}_{d_1 \mathop{|} r} \diamond \psi^{d_2}_2|_{\overline{\mathcal{M}}_{g_2, 2}}\right)\diamond \psi_1^{d_3}|_{\overline{\mathcal{M}}_{g_3,2}} \diamond \overset{\leftharpoondown}{\sfc}^{g_4}_{d_4 \mathop{|} s} 
\\ \notag 
& + (-1)^{k} \sum_{\substack{ g_1+g_2 = g \\ d_1 + d_2 = 2g-1}}  \omega_g\left(\overset{\rightharpoondown}{\sfc}^{g_1}_{d_1 \mathop{|} k} \diamond \psi_2^{d_2}|_{\overline{\mathcal{M}}_{g_2,2}}\right) . 
\end{align}
As in the proof of symmetry, we split the first summand in the following way:
\begin{itemize}
	\item $d_1 + d_2 \leq 2(g_1 + g_2) - 2$ and $d_3 + d_4 \geq 2(g_3 + g_4) - 1$;
	\item $d_1 + d_2 \geq 2(g_1 + g_2) - 1$ and $d_3 + d_4 \leq 2(g_3 + g_4) - 2$.
\end{itemize}
The terms with $d_1 + d_2\leq 2(g_1 + g_2) - 2$ combine in the following way:
\begin{align*}
(-1)^{k} \sum_{\substack{g_1 + g_2 + g_3  = g \\ d_1 + d_2 + d_3  = 2g-2 \\ r + s = k }} (-1)^{d_1} \omega_{g_1} \overset{\rightharpoondown}{\sfc}^{g_1}_{d_1 \mathop{|} r} \diamond \psi_1^{d_2}|_{\overline{\mathcal{M}}_{g_2,2}} \diamond \overset{\leftharpoondown}{\sfc}^{g_3}_{d_3 \mathop{|} s} .
\end{align*}
Note that we do not explicitly impose the condition $r \geq 1$ because we adopt the convention $\omega_0 \overset{\rightharpoondown}{\sfc}^0_{-1 \mathop{|} 0}:=0$. On the other hand, for the terms with $d_3 + d_4\leq 2(g_3 + g_4) - 2$ we obtain
\begin{align*}
(-1)^{k+1} \sum_{\substack{g_1 + g_2 + g_3  = g \\ d_1 + d_2 + d_3  = 2g-2 \\ r + s = k,\, s \geq 1}} (-1)^{d_1 + d_2}  \omega_{g_1+g_2}\left(\overset{\rightharpoondown}{\sfc}^{g_1}_{d_1 \mathop{|} r} \diamond \psi^{d_2}_2|_{\overline{\mathcal{M}}_{g_2, 2}}\right) \diamond \overset{\leftharpoondown}{\sfc}^{g_3}_{d_3 \mathop{|} s}.
\end{align*}
Note that the missing terms with $s = 0$ are exactly the ones in the last line of~\eqref{eq:circlelineEk+1}, \emph{i.e.}, those corresponding to $E_{k+1}$. Putting everything together
\begin{align*}
E_1 + \dots + E_{k+1} &= (-1)^{k} \sum_{\substack{g_1 + g_2 + g_3  = g \\ d_1 + d_2 + d_3  = 2g-2 \\ r + s = k }} (-1)^{d_1} \omega_{g_1} \overset{\rightharpoondown}{\sfc}^{g_1}_{d_1 \mathop{|} r} \diamond \left.\left( \psi_1^{d_2} - (-1)^{d_2} \psi_2^{d_2} \right)\right|_{\overline{\mathcal{M}}_{g_2,2}} \diamond \overset{\leftharpoondown}{\sfc}^{g_3}_{d_3 \mathop{|} s}  
\\
& \quad +  (-1)^{k+1} \sum_{\substack{g_1 + g_2 + g_3 = g \\ d_1 + d_2 + d_3  = 2g-2 \\ r + s = k  }} (-1)^{d_1 + d_2}   \overset{\rightharpoondown}{\sfc}^{g_1}_{d_1 \mathop{|} r} \diamond \omega_{g_2} \psi^{d_2}_2|_{\overline{\mathcal{M}}_{g_2, 2}} \diamond \overset{\leftharpoondown}{\sfc}^{g_3}_{d_3 \mathop{|} s}.
\end{align*}
The result follows from applying \eqref{eq:Liu-Pand-cor2} to the first summand and \eqref{eq:Liu-Pand-cor1-1} to the second one in the expression above. 
\end{proof}
Equation \eqref{eq:omega-B} follows after applying the previous lemma to $E_1 + \dots + E_g=\omega_g B^g$ and noting that the right-hand side of~\eqref{eq:divisors of two types} vanishes in this case because of~\eqref{eq:c-class vanishes}.

\begin{remark}
Note that Properties~\eqref{eq:Thm-irrdiv} and~\eqref{eq:Thm-20div} can be equivalently stated as
$$
\gl_{1*}(\gl_1^*\sfB^g)=0,\qquad \gl_{2*}(\gl_2^*\sfB^g)=0, 
$$
where $\gl_1\colon\oM_{g-1,4}\to\oM_{g,2}$ is the gluing map identifying the last two marked points on a curve from~$\oM_{g-1,4}$, and $\gl_2\colon\oM_{g_1,3}\times\oM_{g_2,1}\to\oM_{g,2}$ is the map gluing the third marked point on a curve from~$\oM_{g_1,3}$ with the marked point on a curve from $\oM_{g_2,1}$. Actually, the arguments from this section can be slightly modified in order to show that $\gl_1^*\sfB^g=0$ and $\gl_2^*\sfB^g=0$, which is stronger than what we have proved. The corresponding properties for the DR cycle are clearly true: $\gl_1^*(\DR_g(1,-1)\lambda_g)=0$ and $\gl_2^*(\DR_g(1,-1)\lambda_g)=0$. This observation belongs to the anonymous referee of our paper and we thank him for sharing it with us.
\end{remark}

\subsection{Intersection with a divisor of curves with marked points on different components} 

The goal of this section is to prove Equation~\eqref{eq:Thm-Inter-1-g1-g2-2}. To this end, we need a new notation. Let $g> h> g_1$, $k\geq 1$. Denote 
\begin{align*}
	\overset{\rightharpoondown}{\sfb}^{h}_{d|k} \coloneqq \sum_{\substack{d_1+d_2 = d}} \sum_{m=1}^{g_1} (-1)^m \overset{\rightharpoondown}{\sfc}^{g_1}_{d_1\mathop{|}m}
	\diamond \left( \sum_{\substack{i_1,\dots,i_k\\a_1,\dots,a_k}} \sfc^{i_1,\dots,i_k}_{a_1,\dots,a_k} + \psi_1 \sum_{\substack{j_1,\dots,j_k\\b_1,\dots,b_k}} \sfc^{j_1,\dots,j_k}_{b_1,\dots,b_k}
	\right), 
\end{align*}
where the first sum in the parentheses is taken over all $i_1,\dots,i_k\geq 1$ such that $i_1+\cdots+i_k=h-g_1$ and all $a_1,\dots,a_k\geq 0$ such that $a_1+\cdots+a_k+k = d_2$  and for any $\ell = 1,\dots,k$ we have 
\[d_1+a_1+\cdots+a_\ell+\ell \leq 2(g_1+i_1+\cdots+i_\ell).\]
The second sum in the parentheses is taken over all $j_1,\dots,j_k\geq 1$ such that $j_1+\cdots+j_k=h-g_1$ and all $b_1,\dots,b_k\geq 0$ such that $b_1+\cdots+b_k+k+1 = d_2$  and for any $\ell = 1,\dots,k$ we have
\[d_1+b_1+\cdots+b_\ell+\ell \leq 2(g_1+j_1+\cdots+j_\ell)-1.\]
In particular, 
\begin{equation}\label{eq:a-class vanishes}
\overset{\rightharpoondown}{\sfb}^{h}_{d|k} =0\quad \text{if}\quad \text{$k>h-g_1$ or $d>2h$},
\end{equation}
and
\begin{align*}
&\overset{\rightharpoondown}{\sfb}^{h}_{d|1} = \sum_{\substack{d_1+d_2 = d}} \sum_{m=1}^{g_1} (-1)^m \overset{\rightharpoondown}{\sfc}^{g_1}_{d_1\mathop{|}m}\diamond \left.\left(\psi_2^{d_2-1}+\psi_1\psi_2^{d_2-2}
	\right)\right|_{\oM_{h-g_1,2}}, && d\le 2g,\\
& \overset{\rightharpoondown}{\sfb}^{h}_{d|k+1}=\sum_{\substack{g_1<f<h\\d_1+d_2=d-1}}\overset{\rightharpoondown}{\sfb}^{f}_{d_1|k}\diamond\psi_2^{d_2}|_{\oM_{h-f,2}},&& d\le 2g, \ k\ge 1.	
\end{align*}

It is convenient to set 
$\overset{\rightharpoondown}{\sfb}^{h}_{d|k} :=0$ for $h\leq g_1$.

\begin{lemma}
We have:
\begin{equation}
  \label{eq:dividepts-pf-ExcessiveIntersection}
  \begin{split}
    \sfB^g \cdot \vcenter{
      \xymatrix@C=10pt@R=10pt{
        &*+[o][F-]{{g_1}}\ar@{-}[l]*{{\ }_1\ }\ar@{-}[r]
        &*+[o][F-]{{g_2}} \ar@{-}[r]*{{\ }_2}
        &}}
    &= \sfB^{g_1}\diamond \sfB^{g_2}
    + \sum_{\substack{d_1+d_2 = 2g}} \sum_{m=1}^{g_1}
    (-1)^{m+1}
    \overset{\rightharpoondown}{\sfc}^{g_1}_{d_1\mathop{|}m} 
    \diamond \left.\left(\psi_2^{d_2} + \psi_1\psi_2^{d_2-1}\right)
    \right|_{\oM_{g_2,2}} \\
    &\quad + \sum_{\substack{g_1<h<g \\ d_1+d_2 = 2g}}
    \sum_{k=1}^{g_2-1} (-1)^{k+1}
    \overset{\rightharpoondown}{\sfb}^{h}_{d_1\mathop{|}k} 
    \diamond \psi_2^{d_2}|_{\oM_{g-h,2}}.    
  \end{split}
\end{equation}
\end{lemma} 

\begin{proof}
  This lemma follows directly from the excess intersection
formula~\cite[Section~A.4]{GraPand}. Let us consider $g=g_1+g_2=f_1+\cdots+f_m$
for some $m\geq 1$. We have:
\begin{equation}
  \label{eq:ExIntersection}
  \begin{split}
    \sfc^{f_1,\dots,f_m}_{a_1,\dots,a_m}
    &\cdot \vcenter{
    \xymatrix@C=10pt@R=10pt{
      &*+[o][F-]{{g_1}}\ar@{-}[l]*{{\ }_1\ }\ar@{-}[r]
      & *+[o][F-]{{g_2}} \ar@{-}[r]*{{\ }_2}
      & }} \\ 
  &= \begin{cases}
    \sfc^{f_1,\dots,f_{i-1},f_i', f_i'', f_{i+1}, \dots,
      f_m}_{a_1,\dots,a_{i-1}, 0, a_i,a_{i+1},\dots,a_m},
    \quad\text{if } g_1 = f_1+\cdots+f_{i-1}+f_i' \text{ and} \,
    f_i=f_i'+f_i'', \text{ where } f_i', f_i''\geq 1 \\ \\
    -\sfc^{f_1,\dots,f_m}_{a_1,\dots,a_{i-1},a_i+1,a_{i+1},\dots,a_m}
    - \sfc^{f_1,\dots,f_i}_{a_1,\dots,a_{i}}
    \diamond \psi_1 \sfc^{f_{i+1},\dots,f_m}_{a_{i+1},\dots,a_{m}},
    \quad\text{if } f_1+\cdots+f_{i}= g_1.
  \end{cases}  
\end{split}
\end{equation}
Recall that in Formula~\eqref{eq:Bg-definition} for $\sfB^g$ we
have only $\sfc^{f_1,\dots,f_m}_{a_1,\dots,a_m}$ satisfying the
conditions 
\[
  a_1+\cdots+a_i+i-1\leq 2(f_1+\cdots+f_i)-1 
\]
for $i=1,\dots,m-1$ and $a_1+\cdots+a_m+m-1=2g$. We apply
Equation~\eqref{eq:ExIntersection} to all terms of the formula for
$\sfB^g$ and we distinguish the following cases:
\begin{enumerate}
\item
  There exists $i$ such that $f_1+\cdots+f_i=g_1$ and in addition $a_1+\cdots+a_i+i-1 = 2(f_1+\cdots+f_i)-1$. The first summands in~\eqref{eq:ExIntersection} applied to these terms form $\sfB^{g_1}\diamond \sfB^{g_2}$.
\item
  There exists $i$ such that $f_1+\cdots+f_i=g_1$ and in addition $a_1+\cdots+a_i+i-1 < 2(f_1+\cdots+f_i)-1$. The first summands in~\eqref{eq:ExIntersection} applied to these terms contribute either to the second (if $i=m-1$) or the third line (if $i<m-1$) of~\eqref{eq:dividepts-pf-ExcessiveIntersection}. More precisely, we can say that in both cases we get terms of the type $\sfc^{j_1,\dots,j_p}_{t_1,\dots,t_p}$ such that $j_1+\cdots+j_q = g_1$ for some $q<p$, with an extra requirement that $t_q> 0$. If $q=p-1$ (respectively, $q<p-1$), these terms land in the second (respectively, third) line of~\eqref{eq:dividepts-pf-ExcessiveIntersection}.
\item
  We have $f_1+\cdots+f_{q-1}<g_1<f_1+\cdots+f_{q}$ for some $1\leq q\leq m$. Apply~\eqref{eq:ExIntersection}. We get exactly the same terms as in the previous case, but now with an extra requirement that $t_q=0$. This and the previous cases deliver together all terms in the second and the third lines of~\eqref{eq:dividepts-pf-ExcessiveIntersection} that do not contain $\psi_1$.
\item
  There exists $i$ such that $f_1+\cdots+f_i=g_1$ and $a_1+\cdots+a_i+i-1 \leq  2(f_1+\cdots+f_i)-1$. The second summands in~\eqref{eq:ExIntersection} applied to these terms form the summands with $\psi_1$ in the second and the third lines of~\eqref{eq:dividepts-pf-ExcessiveIntersection}.
\end{enumerate}
\end{proof}

So our goal is to prove that the sum of the second and the third lines of Equation~\eqref{eq:dividepts-pf-ExcessiveIntersection} vanishes. To this end, we have a more general statement. Let 
\begin{align*}
E: = \sfB^g \cdot \vcenter{
	\xymatrix@C=10pt@R=10pt{
		&*+[o][F-]{{g_1}}\ar@{-}[l]*{{\ }_1\ }\ar@{-}[r]& *+[o][F-]{{g_2}} \ar@{-}[r]*{{\ }_2}& 
	}
} - \sfB^{g_1}\diamond \sfB^{g_2}=\sum_{\ell\ge 1}E_\ell,
\end{align*}
where
\begin{align*}
E_\ell :=   \sum_{\substack{d_1+d_2 = 2g}} \sum_{m=1}^{g_1} (-1)^{m+1} \overset{\rightharpoondown}{\sfc}^{g_1}_{d_1\mathop{|}m}
\diamond \left.\left( \psi_2^{d_2} + \psi_1\psi_2^{d_2-1}\right)\right|_{\oM_{g_2,2}} \cdot \delta_{\ell, 1}+ \sum_{\substack{g_1<h<g \\ d_1+d_2 = 2g}} (-1)^{\ell+1} \overset{\rightharpoondown}{\sfb}^{h}_{d_1\mathop{|} \ell}
\diamond \psi_2^{d_2}|_{\oM_{g-h,2}}.
\end{align*}

\begin{lemma}
  \label{lem:dividepts-induction}
For any $\ell \geq 1$ we have:
\begin{equation}
    \label{eq:sumofEl}
  \begin{split} 
    E_1 + \dots + E_\ell
    &= (-1)^{\ell}
    \sum_{\substack{f_1+f_2=g_2\\d_1+d_2+d_3 = 2g-1}} (-1)^{d_3}
    \sum_{m=1}^{g_1} (-1)^m
    \overset{\rightharpoondown}{\sfc}^{g_1}_{d_1\mathop{|}m} 
    \diamond \left.\left(\psi_2^{d_2} +  \psi_1\psi_2^{d_2-1}\right)
    \right|_{\oM_{f_1,2}}
    \diamond \overset{\leftharpoondown}{\sfc}^{f_2}_{d_3 \mathop{|}
      \ell} \\
    & \quad + (-1)^{\ell}
    \sum_{\substack{g_1<h<g\\ f_1+f_2=g-h \\ d_1+d_2+d_3 = 2g-1}}
    (-1)^{d_3} \sum_{k=1}^{\ell}
    \overset{\rightharpoondown}{\sfb}^{h}_{d_1\mathop{|}k} 
    \diamond\left.\left(\psi_2^{d_2}-(-1)^{d_2}
        \psi_1^{d_2}\right)\right|_{\oM_{f_1,2}}\diamond
    \overset{\leftharpoondown}{\sfc}^{f_2}_{d_3 \mathop{|} \ell-k}.
  \end{split}
\end{equation}
\end{lemma}

\begin{proof}
We prove the lemma by induction. The base of induction is $\ell=1$, and it is equivalent to the following equation:
\begin{equation}
  \label{eq:dividepts-ind1}
  \begin{split}
    \sum_{\substack{d_1+d_2 = 2g}}\,
    \sum_{m=1}^{g_1}
    &\,(-1)^{m+1} \overset{\rightharpoondown}
    {\sfc}^{g_1}_{d_1\mathop{|}m} \diamond \left.
      \left( \psi_2^{d_2} + \psi_1\psi_2^{d_2-1}\right)
    \right|_{\oM_{g_2,2}} \\
    &= - \sum_{\substack{f_1+f_2=g_2\\d_1+d_2+d_3 = 2g-1}}
    (-1)^{d_3} \sum_{m=1}^{g_1} (-1)^m
    \overset{\rightharpoondown}{\sfc}^{g_1}_{d_1\mathop{|}m} 
    \diamond \left.\left(\psi_2^{d_2} + \psi_1\psi_2^{d_2-1}\right)
    \right|_{\oM_{f_1,2}}\diamond
    \overset{\leftharpoondown}{\sfc}^{f_2}_{d_3 \mathop{|} 1} \\
    &\quad - \sum_{\substack{g_1<h<g \\ d_1+d_2 = 2g}} (-1)^{d_2}
    \overset{\rightharpoondown}{\sfb}^{h}_{d_1\mathop{|}1} 
	\diamond\psi_1^{d_2}|_{\oM_{g-h,2}} .
  \end{split}
\end{equation}
We rewrite $\psi_2^{d_2} + \psi_1\psi_2^{d_2-1}$ in the first line of~\eqref{eq:dividepts-ind1} as 
\begin{align*}
	   (-1)^{d_2-1}\psi_1^{d_2}+\psi_2^{d_2} - \psi_1\left((-1)^{d_2-1}\psi_1^{d_2-1} -\psi_2^{d_2-1} \right).
\end{align*}
Noting that $d_1\leq 2g_1-1$ implies $d_2-1\geq  2g_2$, we apply identity~\eqref{eq:Liu-Pand-cor2} twice to obtain
\begin{equation}\label{eq:psi1psi2 identity}
\left.\left(\psi_2^{d_2} + \psi_1\psi_2^{d_2-1}\right)\right|_{\oM_{g_2,2}}=\sum_{\substack{f_1+f_2=g_2\\a_1+a_2=d_2-1}} (-1)^{a_2} \left.\left(\psi_1\psi_2^{a_1-1}+\psi_2^{a_1}\right)\right|_{\oM_{f_1,2}} \diamond \psi_1^{a_2}|_{\oM_{f_2,2}},\quad d_2\ge 2g_2+1.
\end{equation}
If $a_2\leq 2f_2-1$ (in both summands), then we obtain the second line in~\eqref{eq:dividepts-ind1}, and if $a_2\geq 2f_2$, then we obtain the third line in~\eqref{eq:dividepts-ind1}.

The induction step is equivalent to the following equation:
\begin{align} 
  (-1)^\ell
  & \sum_{\substack{f_1+f_2=g_2\\d_1+d_2+d_3 = 2g-1}}
  (-1)^{d_3} \sum_{m=1}^{g_1} (-1)^{m}
  \overset{\rightharpoondown}{\sfc}^{g_1}_{d_1\mathop{|}m} 
  \diamond \left.\left(\psi_2^{d_2} +
  \psi_1\psi_2^{d_2-1}\right)\right|_{\oM_{f_1,2}}
  \diamond \overset{\leftharpoondown}{\sfc}^{f_2}_{d_3 \mathop{|}
  \ell}\label{eq:dividepts-pf-indstep,line 1}\\  
  &\quad + (-1)^\ell \sum_{\substack{g_1<h<g \\
  f_1+f_2=g-h \\
  d_1+d_2+d_3 = 2g-1}} (-1)^{d_3} \sum_{k=1}^{\ell}
  \overset{\rightharpoondown}{\sfb}^{h}_{d_1\mathop{|}k} 
  \diamond  \left.\left( \psi_2^{d_2} - (-1)^{d_2}
  \psi_1^{d_2}\right)\right|_{\oM_{f_1,2}} \diamond
  \overset{\leftharpoondown}{\sfc}^{f_2}_{d_3 \mathop{|}
  \ell-k} \label{eq:dividepts-pf-indstep,line 2}\\ 
  &= (-1)^{\ell +1} \sum_{\substack{f_1+f_2=g_2\\d_1+d_2+d_3 = 2g-1}}
  (-1)^{d_3} \sum_{m=1}^{g_1} (-1)^{m}
  \overset{\rightharpoondown}{\sfc}^{g_1}_{d_1\mathop{|}m} 
	\diamond \left.\left( \psi_2^{d_2} +
  \psi_1\psi_2^{d_2-1}\right)\right|_{\oM_{f_1,2}}\diamond
  \overset{\leftharpoondown}{\sfc}^{f_2}_{d_3 \mathop{|}
  \ell+1} \label{eq:dividepts-pf-indstep,line 3} \\
  &\quad - (-1)^{\ell + 1} \sum_{\substack{g_1<h<g\\ f_1+f_2=g-h \\
  d_1+d_2+d_3 = 2g-1}} (-1)^{d_3} \sum_{k=1}^{\ell + 1}
  \overset{\rightharpoondown}{\sfb}^{h}_{d_1\mathop{|}k} 
  \diamond  (-1)^{d_2} \psi_1^{d_2}|_{\oM_{f_1,2}}
  \diamond \overset{\leftharpoondown}{\sfc}^{f_2}_{d_3 \mathop{|}
  \ell+1-k} \label{eq:dividepts-pf-indstep,line 4}  \\	
  &\quad + (-1)^{\ell + 1} \sum_{\substack{g_1<h<g\\ f_1+f_2=g-h \\
  d_1+d_2+d_3 = 2g-1}} (-1)^{d_3} \sum_{k=1}^{\ell + 1}
  \overset{\rightharpoondown}{\sfb}^{h}_{d_1\mathop{|}k} 
  \diamond   \psi_2^{d_2}|_{\oM_{f_1,2}}
  \diamond \overset{\leftharpoondown}{\sfc}^{f_2}_{d_3 \mathop{|}
  \ell+1-k} \label{eq:dividepts-pf-indstep,line 5}\\
  &\quad + (-1)^{\ell+1} \sum_{\substack{g_1<h<g \\ d_1+d_2 = 2g}}
  \overset{\rightharpoondown}{\sfb}^{h}_{d_1\mathop{|}\ell+1} \diamond
  \psi_2^{d_2}|_{\oM_{g-h,2}},\label{eq:dividepts-pf-indstep,line 6} 
\end{align}
where $\ell\ge 1$. 

In line~\eqref{eq:dividepts-pf-indstep,line 1} we have $d_1\leq 2g_1-1$ and $d_3\leq 2f_2-1$, hence $d_2-1\geq 2f_1$ and by~\eqref{eq:psi1psi2 identity} the expression in line~\eqref{eq:dividepts-pf-indstep,line 1} is equal to
$$
(-1)^\ell \sum_{\substack{f_1+f_2=g_2\\d_1+d_2+d_3 = 2g-1}}\sum_{\substack{h_1+h_2=f_1\\a_1+a_2=d_2-1}} \sum_{m=1}^{g_1}(-1)^{m} (-1)^{a_2+d_3} \overset{\rightharpoondown}{\sfc}^{g_1}_{d_1\mathop{|}m}
	\diamond \left.\left(\psi_2^{a_1} + \psi_1\psi_2^{a_1-1}\right)\right|_{\oM_{h_1,2}}\diamond \psi_1^{a_2}|_{\oM_{h_2,2}}\diamond \overset{\leftharpoondown}{\sfc}^{f_2}_{d_3 \mathop{|} \ell}.
$$
The part of this sum with $a_2+d_3\le 2(h_2+f_2)-2$ is equal to the expression in line~\eqref{eq:dividepts-pf-indstep,line 3} while the part with $a_2+d_3\ge 2(h_2+f_2)-1$ is equal to the $k=1$ term of the expression in line~\eqref{eq:dividepts-pf-indstep,line 4}.

In line~\eqref{eq:dividepts-pf-indstep,line 2} we have $d_1\leq 2h$ and $d_3\leq 2f_2-1$, hence $d_2\geq 2f_1$ and applying identity~\eqref{eq:Liu-Pand-cor2} we get
$$
(-1)^\ell \sum_{\substack{g_1<h<g\\ f_1+f_2=g-h \\ d_1+d_2+d_3 = 2g-1}} \sum_{\substack{a_1+a_2=d_2-1\\h_1+h_2=f_1}}\sum_{k=1}^{\ell}(-1)^{a_2+d_3} \overset{\rightharpoondown}{\sfb}^{h}_{d_1\mathop{|}k}
	\diamond \psi_2^{a_1}|_{\oM_{h_1,2}}\diamond\psi_1^{a_2}|_{\oM_{h_2,2}} \diamond \overset{\leftharpoondown}{\sfc}^{f_2}_{d_3 \mathop{|} \ell-k}.
$$
The part of this sum with $a_2+d_3\le 2(h_2+f_2)-2$ is equal to the part of~\eqref{eq:dividepts-pf-indstep,line 5} with $k=1,\dots,\ell$ while the part with $a_2+d_3\ge 2(h_2+f_2)-1$ is equal to the part of~\eqref{eq:dividepts-pf-indstep,line 4} with $k=2,\dots,\ell+1$.

Finally the part of \eqref{eq:dividepts-pf-indstep,line 5} with $k=\ell+1$ is equal exactly to~\eqref{eq:dividepts-pf-indstep,line 6} with the opposite sign. This completes the proof of the induction step and the proof of the lemma.
\end{proof}
Equation \eqref{eq:Thm-Inter-1-g1-g2-2} follows after applying the above lemma to $E = E_1 + \dots + E_{g_2}$ and noting that the right-hand side of~\eqref{eq:sumofEl} vanishes for $\ell=g_2$ because of~\eqref{eq:c-class vanishes} and~\eqref{eq:a-class vanishes}.

\subsection{Evaluation of psi class on a pull-back}

To prove \eqref{eq:thm-pullback}, let us introduce the notation
\begin{align*}
\overset{\leftharpoondown}{\sfd}^g_{d \mathop{|} k} := \sum_{\ell = 1}^{k} \sum_{\substack {g_1, \dots, g_k \\ d_1, \dots, d_k}} \psi_1^{d_k}|_{\overline{\mathcal{M}}_{g_k, 2}} \diamond \psi_1^{d_{k-1}}|_{\overline{\mathcal{M}}_{g_{k-1}, 2}} \diamond \dots \diamond \psi_1^{d_\ell}|_{\overline{\mathcal{M}}_{g_\ell, 3}} \diamond \dots \diamond \psi_1^{d_1}|_{\overline{\mathcal{M}}_{g_1, 2}},\quad k\ge 1,
\end{align*}
where the sum is taken over all $g_1, \dots, g_k \geq 1$ and $d_1, \dots, d_k \geq 0$ satisfying $g_1 + \dots + g_k = g$, $d_1 + \dots + d_k + k - 1 = d$, and $d_1 + \dots + d_s + s - 1 \leq 2(g_1 + \dots + g_s) - 1$ for all $1 \leq s \leq k$. Similarly, we define
\begin{align*}
\overset{\leftharpoondown}{\sfe}^g_{d \mathop{|} k}  := \sum_{\ell = 1}^{k} \sum_{\substack {g_1, \dots, g_k \\ d_1, \dots, d_k}} \psi_1^{d_k}|_{\overline{\mathcal{M}}_{g_k, 2}} \diamond  \dots \diamond \psi_1^{d_{\ell+1}}|_{\overline{\mathcal{M}}_{g_{\ell+1}, 2}} \diamond 1|_{\overline{\mathcal{M}}_{0, 3}} \diamond \psi_1^{d_{\ell-1}}|_{\overline{\mathcal{M}}_{g_{\ell - 1}, 2}} \diamond \dots \diamond \psi_1^{d_1}|_{\overline{\mathcal{M}}_{g_1, 2}},\quad k\ge 1,
\end{align*}
where the sum is taken over all $g_1, \dots, g_{\ell-1}, g_{\ell + 1}, \dots, g_k \geq 1$ and $d_1, \dots, d_k \geq 0$ satisfying $g_1 + \dots + g_k = g$, $d_1 + \dots + d_k + k - 1 = d$, and $d_1 + \dots + d_s + s - 1 \leq 2(g_1 + \dots + g_s) - 1$ for all $1 \leq s \leq k$. Note that in particular $d_\ell = g_\ell = 0$. Note also that $\overset{\leftharpoondown}{\sfe}^g_{d \mathop{|} 1}=0$. We will adopt the convention $\overset{\leftharpoondown}{\sfd}^g_{d \mathop{|} 0}=\overset{\leftharpoondown}{\sfe}^g_{d \mathop{|} 0}:=0$. 

Using 
\begin{align*}
\pi^*(\psi_1^a|_{\overline{\mathcal{M}}_{g,2} }) = \psi_1^a|_{\overline{\mathcal{M}}_{g,3} } - 1|_{\overline{\mathcal{M}}_{0,3}} \diamond \psi_1^{a-1}|_{\overline{\mathcal{M}}_{g,2}},\quad a\ge 1,
\end{align*}
it is straightforward to see that 
\begin{align}\label{eq:pullback of c}
\pi^*(\overset{\leftharpoondown}{\sfc}^g_{d \mathop{|} k}) = \overset{\leftharpoondown}{\sfd}^g_{d \mathop{|} k}
 - \overset{\leftharpoondown}{\sfe}^g_{d \mathop{|} k+1}.
\end{align}
As before, let
\begin{align}
  E
  &:= \pi^*(\sfB^g) \psi_1 - \sfB^g \diamond
    1|_{\overline{\mathcal{M}}_{0,3}} - \sum_{g_1 + g_2 = g}
    \sfB^{g_1} \diamond \pi^*(\sfB^{g_2})\notag \\
  &\phantom{:}= \sum_{\substack { g_1+g_2 = g \\ d_1 + d_2 = 2g-1}}
  \sum_{k=0}^{g_2+1} (-1)^k \left(
  \psi_1^{d_1+1}|_{\overline{\mathcal{M}}_{g_1,3}} \diamond
  \overset{\leftharpoondown}{\sfc}^{g_2}_{d_2 \mathop{|} k} +
  \psi_1^{d_1+1}|_{\overline{\mathcal{M}}_{g_1,2}}
  \diamond\left(\overset{\leftharpoondown}{\sfd}^{g_2}_{d_2 \mathop{|}
  k}+\overset{\leftharpoondown}{\sfe}^{g_2}_{d_2 \mathop{|} k}\right)
  \right) \notag\\ 
  & \quad - \sum_{\substack { g_1+g_2 = g \\ d_1 + d_2 = 2g-1}}
  \sum_{k=0}^{g_1} (-1)^k \overset{\rightharpoondown}{\sfc}^{g_1}_{d_1
  \mathop{|} k} \diamond \psi_2^{d_2}|_{\overline{\mathcal{M}}_{g_2,
  2}} \diamond 1|_{\overline{\mathcal{M}}_{0,3}}
  \label{eq:with0ands0} \\
  & \quad - \sum_{\substack{g_1 + g_2 + g_3 + g_4 = g \\  d_1 + d_2 =
  2(g_1+g_2) - 1 \\  d_3 + d_4 = 2(g_3 + g_4) - 1 \\ s \geq 1 \
  \textrm{if} \ g_3 = 0}}  \sum_{r=0}^{g_1} \sum_{s=0}^{g_4}
  (-1)^{r+s} \overset{\rightharpoondown}{\sfc}^{g_1}_{d_1 \mathop{|}
  r} \diamond \psi_2^{d_2}|_{\overline{\mathcal{M}}_{g_2, 2}} \diamond
  \psi_1^{d_3}|_{\overline{\mathcal{M}}_{g_3,3}} \diamond
  \overset{\leftharpoondown}{\sfc}^{g_4}_{d_4 \mathop{|}
  s} \label{eq:summandwiththreelegs} \\ 
& \quad - \sum_{\substack{g_1 + g_2 + g_3 + g_4 = g \\  d_1 + d_2 =
  2(g_1+g_2) - 1 \\  d_3 + d_4 = 2(g_3 + g_4) - 1 }}  \sum_{r=0}^{g_1}
  \sum_{s=0}^{g_4+1} (-1)^{r+s}
  \overset{\rightharpoondown}{\sfc}^{g_1}_{d_1 \mathop{|} r} \diamond
  \psi_2^{d_2}|_{\overline{\mathcal{M}}_{g_2, 2}} \diamond
  \psi_1^{d_3}|_{\overline{\mathcal{M}}_{g_3,2}} \diamond
  \left(\overset{\leftharpoondown}{\sfd}^{g_4}_{d_4 \mathop{|}
  s}+\overset{\leftharpoondown}{\sfe}^{g_4}_{d_4 \mathop{|}
  s}\right).\notag 
\end{align}
where we have used the already proven symmetry \eqref{eq:symmetric} and the corresponding mirror formula of \eqref{eq:Bg-definition}
\begin{align*}
\sfB^g = \sum_{\substack{g_1+g_2=g \\ d_1+d_2=2g-1}}  \sum_{k=0}^{g_2} (-1)^k \psi_1^{d_1}|_{\oM_{g_1,2}}
\diamond \overset{\leftharpoondown}{\sfc}^{g_2}_{d_2\mathop{|}k}
\end{align*}for the factors on the right-hand side of $\diamond$. Note that \eqref{eq:with0ands0} is exactly the forbidden case $g_3 = 0$ and $s = 0$ in \eqref{eq:summandwiththreelegs}. Let $E_k$ denote the terms in the expression above that have $k$ components, \emph{i.e.},
\begin{align}
  E_k &:= \notag\\ 
  &(-1)^{k-1} \sum_{\substack { g_1+g_2 = g \\ d_1 + d_2 = 2g-1}}
  \left(  \psi_1^{d_1+1}|_{\overline{\mathcal{M}}_{g_1,3}} \diamond
  \overset{\leftharpoondown}{\sfc}^{g_2}_{d_2 \mathop{|} k-1} +
  \psi_1^{d_1+1}|_{\overline{\mathcal{M}}_{g_1,2}} \diamond
  \left(\overset{\leftharpoondown}{\sfd}^{g_2}_{d_2 \mathop{|}
  k-1}+\overset{\leftharpoondown}{\sfe}^{g_2}_{d_2 \mathop{|}
  k-1}\right)\right) \label{eq:r=0pullback} \\ 
  & + (-1)^{k-1}\hspace{-0.5cm}
    \sum_{\substack{%
    g_1 + g_2 + g_3 + g_4 = g \\
  d_1 + d_2 = 2(g_1+g_2) - 1 \\
  d_3 + d_4 = 2(g_3 + g_4) - 1 \\
  r+s = k-2}}\hspace{-0.4cm}
  \overset{\rightharpoondown}{\sfc}^{g_1}_{d_1
  \mathop{|} r} \diamond \psi_2^{d_2}|_{\overline{\mathcal{M}}_{g_2,
  2}}  \diamond \left(\psi_1^{d_3}|_{\overline{\mathcal{M}}_{g_3, 3}}
  \diamond \overset{\leftharpoondown}{\sfc}^{g_4}_{d_4 \mathop{|}
  s}+\psi_1^{d_3}|_{\overline{\mathcal{M}}_{g_3, 2}} \diamond
  \left(\overset{\leftharpoondown}{\sfd}^{g_4}_{d_4 \mathop{|}
  s}+\overset{\leftharpoondown}{\sfe}^{g_4}_{d_4 \mathop{|}
  s}\right)\right). \label{eq:123kind} 
\end{align}
The following inductive lemma will immediately imply Equation
\eqref{eq:thm-pullback}.

\begin{lemma}
We can write $E_1 + E_2 + \dots + E_k$ as an expression involving only graphs with $k+1$ vertices. More precisely, we have:
\begin{align*}
  E_1 +
  & E_2 + \cdots + E_k \\
  &= (-1)^{k} \hspace{-0.4cm}\sum_{\substack { g_1+g_2+g_3 + g_4 = g
  \\ d_1 + d_2 + d_3 + d_4 = 2g-2 \\ r+s = k-1}}
  \hspace{-0.25cm}(-1)^{d_1+d_2}\overset{\rightharpoondown}{\sfc}^{g_1}_{d_1
  \mathop{|} r} \diamond \psi_2^{d_2}|_{\overline{\mathcal{M}}_{g_2,
  2}}  \diamond \left(\psi_1^{d_3}|_{\overline{\mathcal{M}}_{g_3, 2}}
  \diamond \left(\overset{\leftharpoondown}{\sfd}^{g_4}_{d_4
  \mathop{|} s}+\overset{\leftharpoondown}{\sfe}^{g_4}_{d_4 \mathop{|}
  s}\right)+\psi_1^{d_3}|_{\overline{\mathcal{M}}_{g_3, 3}} \diamond
  \overset{\leftharpoondown}{\sfc}^{g_4}_{d_4 \mathop{|} s}\right). 
\end{align*}
\end{lemma}
\begin{proof}
We proceed by induction on $k$. The case $k=1$ is clear, as
\begin{align*}
  E_1
  &= \psi_1^{2g+1}|_{\overline{\mathcal{M}}_{g,3}}
    = \sum_{\substack{ g_1+g_2 = g \\ d_1 + d_2 = 2g}} (-1)^{d_1}
  \psi_2^{d_1}|_{\overline{\mathcal{M}}_{g_1, 2}} \diamond
  \psi_1^{d_2}|_{\overline{\mathcal{M}}_{g_2,3}} \\
  &= \sum_{\substack { g_2 + g_3 = g \\ d_2 + d_3 = 2g}} (-1)^{d_2}
  \overset{\rightharpoondown}{\sfc}^0_{-1 \mathop{|} 0} \diamond
  \psi_2^{d_2}|_{\overline{\mathcal{M}}_{g_2, 2}} \diamond
  \psi_1^{d_3}|_{\overline{\mathcal{M}}_{g_3,3}} \diamond
  \overset{\leftharpoondown}{\sfc}^0_{-1 \mathop{|} 0} 
\end{align*}
by \eqref{eq:Liu-Pand-cor1-2}.

Assume the lemma holds for $k \geq 1$, then we split $E_1 + \dots + E_k$ into three kinds of summands, according to the powers of the psi classes:
\begin{itemize}
\item $d_1 + d_2 < 2(g_1 + g_2) - 1$ and $d_3 + d_4 > 2(g_3 + g_4) - 1$;
\item $d_1 + d_2 > 2(g_1 + g_2) - 1$ and $d_3 + d_4 < 2(g_3 + g_4) - 1$;
\item $d_1 + d_2 = 2(g_1 + g_2) - 1$ and $d_3 + d_4 = 2(g_3 + g_4) - 1$.
\end{itemize}
Note that the summands of the third kind cancel out with \eqref{eq:123kind} for $E_{k+1}$. We rewrite the terms with $d_1 + d_2< 2(g_1 + g_2) - 1$ as
\begin{align*}
  (-1)^{k}
  &\sum_{\substack { g_1+g_2+g_3 + g_4 = g \\
  d_1 + d_2 + d_3 + d_4 = 2g-2 \\
  r+s = k-1\\
  d_1 + d_2< 2(g_1 + g_2) - 1}}
  (-1)^{d_1+d_2}\overset{\rightharpoondown}{\sfc}^{g_1}_{d_1
  \mathop{|} r} \diamond \psi_2^{d_2}|_{\overline{\mathcal{M}}_{g_2,
  2}}  \diamond \left(\psi_1^{d_3}|_{\overline{\mathcal{M}}_{g_3, 2}}
  \diamond \left(\overset{\leftharpoondown}{\sfd}^{g_4}_{d_4
  \mathop{|} s}+\overset{\leftharpoondown}{\sfe}^{g_4}_{d_4 \mathop{|}
  s}\right)+\psi_1^{d_3}|_{\overline{\mathcal{M}}_{g_3, 3}} \diamond
  \overset{\leftharpoondown}{\sfc}^{g_4}_{d_4 \mathop{|} s}\right) \\ 
  &= (-1)^{k+1} \sum_{\substack { g_1+g_2+g_3 = g \\ d_1 + d_2 + d_3
  = 2g-1 \\ r+s = k,\, r \geq 1}} (-1)^{d_1}
  \overset{\rightharpoondown}{\sfc}^{g_1}_{d_1 \mathop{|} r} \diamond
  \left(\psi_1^{d_2}|_{\overline{\mathcal{M}}_{g_2, 2}} \diamond
  \left(\overset{\leftharpoondown}{\sfd}^{g_3}_{d_3 \mathop{|}
  s}+\overset{\leftharpoondown}{\sfe}^{g_3}_{d_3 \mathop{|} s}\right)+
  \psi_1^{d_2}|_{\overline{\mathcal{M}}_{g_2, 3}} \diamond
  \overset{\leftharpoondown}{\sfc}^{g_3}_{d_3 \mathop{|} s}\right). 
\end{align*}
Note that the expresion in line \eqref{eq:r=0pullback} for $E_{k+1}$
consists exactly of those terms with $r=0$. Similarly, for the terms
with $d_3 + d_4 < 2(g_3 + g_4) - 1$:
\begin{align*} 
  (-1)^{k}
  &\sum_{\substack { g_1+g_2+g_3 + g_4 = g \\ d_1 + d_2 + d_3
  + d_4 = 2g-2 \\ r+s = k-1 \\ d_3 + d_4 < 2(g_3 + g_4) - 1 \\ g_3
  \geq  1}} (-1)^{d_1+d_2}
  \overset{\rightharpoondown}{\sfc}^{g_1}_{d_1 \mathop{|} r} \diamond
  \psi_2^{d_2}|_{\overline{\mathcal{M}}_{g_2, 2}}  \diamond \left(
  \psi_1^{d_3}|_{\overline{\mathcal{M}}_{g_3, 2}} \diamond
  \overset{\leftharpoondown}{\sfd}^{g_4}_{d_4 \mathop{|} s} +
  \psi_1^{d_3}|_{\overline{\mathcal{M}}_{g_3, 3}} \diamond
  \overset{\leftharpoondown}{\sfc}^{g_4}_{d_4 \mathop{|} s} \right) \\ 
  &+(-1)^{k} \sum_{\substack { g_1+g_2+g_3 + g_4 = g \\ d_1 + d_2 + d_3
  + d_4 = 2g-2 \\ r+s = k-1 \\ d_3 + d_4 < 2(g_3 + g_4) - 1}}
  (-1)^{d_1+d_2} \overset{\rightharpoondown}{\sfc}^{g_1}_{d_1
  \mathop{|} r} \diamond \psi_2^{d_2}|_{\overline{\mathcal{M}}_{g_2,
  2}}  \diamond \left( \psi_1^{d_3}|_{\overline{\mathcal{M}}_{g_3, 2}}
  \diamond \overset{\leftharpoondown}{\sfe}^{g_4}_{d_4 \mathop{|} s} +
  \delta_{g_3,0} 1|_{\overline{\mathcal{M}}_{0, 3}} \diamond
  \overset{\leftharpoondown}{\sfc}^{g_4}_{d_4 \mathop{|} s} \right)\\ 
  &= -(-1)^{k+1} \sum_{\substack { g_1+g_2+g_3 = g \\ d_1 + d_2 + d_3  =
  2g-1 \\ r+s = k}} (-1)^{d_1+d_2}
  \overset{\rightharpoondown}{\sfc}^{g_1}_{d_1 \mathop{|} r} \diamond
  \psi_2^{d_2}|_{\overline{\mathcal{M}}_{g_2, 2}}  \diamond
  \left(\overset{\leftharpoondown}{\sfd}^{g_3}_{d_3 \mathop{|} s} +
  \overset{\leftharpoondown}{\sfe}^{g_3}_{d_3 \mathop{|} s}\right). 
\end{align*}
Putting everything together, we get
\begin{align*}
  E_1 + \dots + E_{k+1}
  &= (-1)^{k+1} \sum_{\substack { g_1+g_2+g_3 = g \\ d_1 + d_2 + d_3
  = 2g-1 \\ r+s = k}} (-1)^{d_1}
  \overset{\rightharpoondown}{\sfc}^{g_1}_{d_1 \mathop{|} r} \diamond
  \left.\left(\psi_1^{d_2} - (-1)^{d_2}
  \psi_2^{d_2}\right)\right|_{\overline{\mathcal{M}}_{g_2, 2}}
  \diamond \left(\overset{\leftharpoondown}{\sfd}^{g_3}_{d_3
  \mathop{|} s}+\overset{\leftharpoondown}{\sfe}^{g_3}_{d_3 \mathop{|}
  s}\right) \\
  &\quad + (-1)^{k+1} \sum_{\substack { g_1+g_2+g_3 = g \\ d_1 + d_2 + d_3  = 2g-1 \\ r+s = k }} (-1)^{d_1} \overset{\rightharpoondown}{\sfc}^{g_1}_{d_1 \mathop{|} r} \diamond \psi_1^{d_2}|_{\overline{\mathcal{M}}_{g_2, 3}} \diamond \overset{\leftharpoondown}{\sfc}^{g_3}_{d_3 \mathop{|} s}.
\end{align*}
We apply \eqref{eq:Liu-Pand-cor2} to the first summand and \eqref{eq:Liu-Pand-cor1-2} to the second one to obtain
\begin{align*}
  E_1 + \dots + E_{k+1}
  &= (-1)^{k+1} \sum_{\substack { g_1+g_2+g_3+g_4 = g \\
  d_1 + d_2 + d_3+d_4  = 2g-2 \\
  r+s = k}}
  (-1)^{d_1+d_2} \overset{\rightharpoondown}{\sfc}^{g_1}_{d_1 \mathop{|} r}
  \diamond \psi_2^{d_2}|_{\overline{\mathcal{M}}_{g_2, 2}}
  \diamond  \psi_1^{d_3}|_{\overline{\mathcal{M}}_{g_3, 2}}
  \diamond \left(\overset{\leftharpoondown}{\sfd}^{g_4}_{d_4
  \mathop{|} s}
  + \overset{\leftharpoondown}{\sfe}^{g_4}_{d_4 \mathop{|} s}\right) \\
  &\quad + (-1)^{k+1} \sum_{\substack { g_1+g_2+g_3+g_4 = g \\ d_1 + d_2 +
  d_3+d_4  = 2g-2 \\ r+s = k}} (-1)^{d_1+d_2}
  \overset{\rightharpoondown}{\sfc}^{g_1}_{d_1 \mathop{|} r} \diamond
  \psi_2^{d_2}|_{\overline{\mathcal{M}}_{g_2, 2}} \diamond
  \psi_1^{d_3}|_{\overline{\mathcal{M}}_{g_3, 3}}  \diamond
  \overset{\leftharpoondown}{\sfc}^{g_4}_{d_4 \mathop{|} s}, 
\end{align*}
which concludes the proof. 
\end{proof}
Equation \eqref{eq:thm-pullback} follows from applying the above lemma to $E = E_1 + \dots + E_{g+1}$, again using \eqref{eq:c-class vanishes}.

\subsection{Evaluation of psi class}
Here we prove Equation \eqref{eq:Thm-ev-psiclass}. We derive it from equation~\eqref{eq:thm-pullback}. For $I\subset\{1,\dots,n\}$ with $|I|\ge 2$ denote by $\delta^I_{0}\in R^1(\oM_{g,n})$ the class of the closure of the subset of stable curves from $\oM_{g,n}$ having exactly one node separating a genus~$0$ component carrying the points marked by~$I$ and a genus~$g$ component carrying the points marked by~$\{1,\dots,n\}\setminus I$. Denote by $\pi^{(g)}\colon\oM_{g,3}\to\oM_{g,2}$ the map that forgets the third marked point. Multiplying equation~\eqref{eq:thm-pullback} by $\psi_3$ and taking the push-forward by $\pi^{(g)}$, we obtain
$$
\pi^{(g)}_*\left(\psi_3\psi_1\cdot\pi^{(g)*}\sfB^g\right) =\sum_{\substack{g_1+g_2=g\\ g_1,g_2\geq 1}} \sfB^{g_1}\diamond \pi_*^{(g_2)}\left(\psi_3\cdot\pi^{(g_2)*}\sfB^{g_2}\right).
$$ 
Noting that
$\psi_3\psi_1=\psi_3\left(\pi^{(g)*}\psi_1+\delta_0^{\{1,3\}}\right)
=\psi_3\pi^{(g)*}\psi_1$
and $\pi^{(g)}_*\psi_3=2g$, we get 
\[
  2g\psi_1\cdot\sfB^g =\sum_{\substack{g_1+g_2=g\\ g_1,g_2\geq 1}}
  2g_2\sfB^{g_1}\diamond \sfB^{g_2},
\]
as required.

\subsection{Equivalence of the conjectural formulas}\label{sec:proofequiv}

In this section we prove Theorem~\ref{thm:equivalent}. 

\begin{lemma} \label{lem:BgBg} 
Suppose $C^g=\DR_g(1,-1)\lambda_g$ or $C^g=\sfB^g$. Then for $g\geq 1$ we have
\begin{gather}\label{eq:recursion for C} 
	 C^g=\psi_1\cdot\pi_2^*\pi_{2*}C^g- \sum_{\substack{g_1+g_2=g\\ g_1,g_2\geq 1}}C^{g_1} \diamond \pi_2^*\pi_{2*}C^{g_2},
\end{gather}
where $\pi_2\colon\oM_{h,2}\to\oM_{h,1}$ is the map forgetting the second marked point.
\end{lemma}
\begin{proof}
The proof is based on Properties~\eqref{eq:symmetric}--\eqref{eq:Thm-ev-psiclass}, which are also true for the class $\DR_g(1,-1)\lambda_g$. Consider the following diagram of forgetful maps:
\begin{gather*}
\xymatrix{
\oM_{g,3}\ar[d]_{\pi_3}\ar[r]^{\tpi_2} & \oM_{g,2}\ar[d]^{\hpi_2}\\
\oM_{g,2}\ar[r]^{\pi_2} & \oM_{g,1}
}
\end{gather*}
where the subindices denote the number of the point that a map forgets. Note that under the map $\tpi_2$ the third marked point on a curve from $\oM_{g,3}$ becomes the second marked point on the resulting curve from $\oM_{g,2}$. We then compute
\begin{align*}
  \psi_1\cdot\pi_2^*\pi_{2*}C^g
  &= \psi_1\cdot\pi_2^*\hpi_{2*}C^g
    = \psi_1\cdot\pi_{3*}\tpi^*_2C^g
    = \pi_{3*}\left(\pi_3^*\psi_1\cdot\tpi^*_2C^g\right) \\
  &= \pi_{3*}\left(\left(\psi_1-\delta_0^{\{1,3\}}\right)\cdot\tpi^*_2C^g\right)
    = \pi_{3*}\left(\psi_1\cdot\tpi^*_2C^g\right)-\pi_{3*}
    \left(\delta_0^{\{1,3\}}\cdot\tpi^*_2C^g\right) \\
  &\stackrel{\eqref{eq:thm-pullback}}{=}\pi_{3*}\Bigg(C^g\diamond
    1|_{\oM_{0,3}} +\sum_{\substack{g_1+g_2=g\\ g_1,g_2\geq 1}}
  C^{g_1}\diamond\tpi_2^*C^{g_2}\Bigg)-\pi_{3*}
  \left(\delta_0^{\{1,3\}}\cdot\tpi^*_2C^g\right) \\
  &= C^g+\sum_{\substack{g_1+g_2=g\\ g_1,g_2\geq 1}}
  C^{g_1}\diamond\pi_{3*}\tpi_2^*C^{g_2}
  -\pi_{3*}\left(\delta_0^{\{1,3\}}\cdot\tpi^*_2 C^g\right) \\
  &= C^g
  + \sum_{\substack{g_1+g_2=g\\ g_1,g_2\geq 1}}
  C^{g_1} \diamond \pi_2^*\hpi_{2*}
  C^{g_2}-\pi_{3*}\left(\delta_0^{\{1,3\}}\cdot\tpi^*_2
  C^g\right),
\end{align*}
and it is sufficient to check that $\delta_0^{\{1,3\}}\cdot\tpi^*_2C^g=0$. 

Indeed, for $C_g=\DR_g(1,-1)\lambda_g$ we have 
\[
  \delta_0^{\{1,3\}}\cdot\tpi^*_2(\DR_g(1,-1)\lambda_g)
  = \delta_0^{\{1,3\}}\cdot\DR_g(1,0,-1)\lambda_g
  = 1|_{\oM_{0,3}}\diamond\DR_g(0,0)\lambda_g
  = (-1)^g1|_{\oM_{0,3}}\diamond\lambda_g^2=0.
\]

In the case $C^g=\sfB^g$, from~\eqref{eq:pullback of c} it is easy to
see that the class $\tpi^*_2\sfB^g-\psi_1^{2g}$ is supported on the
stratum in $\oM_{g,3}$ that doesn't intersect the divisor
corresponding to~$\delta_0^{\{1,3\}}$. Therefore,
$\delta_0^{\{1,3\}}\cdot(\tpi^*_2\sfB^g-\psi_1^{2g})=0$, but we also
obviously have $\delta_0^{\{1,3\}}\cdot\psi_1^{2g}=0$, which gives
$\delta_0^{\{1,3\}}\cdot\tpi^*_2\sfB^g=0$. 
\end{proof}

\begin{proof}[Proof of Theorem~\ref{thm:equivalent}.]
Assuming $\pi_{2*}\left(\DR_g(1,-1)\lambda_g\right)=\pi_{2*}\sfB^g$,
the equality $\DR_g(1,-1)\lambda_g=\sfB^g$ immediately follows from
Formula~\eqref{eq:recursion for C} by the induction on $g$. This
completes the proof of the theorem. 
\end{proof}


\end{document}